\documentclass[11pt]{amsart}
\usepackage{amsmath}
\usepackage{amssymb}
\usepackage{amscd}
\usepackage{enumerate}
\usepackage{verbatim}
\usepackage{color}

\newtheorem{Thm}{Theorem}[section]
\newtheorem{Lemma}[Thm]{Lemma}

\newtheorem{Prop}[Thm]{Proposition}
\theoremstyle{definition}

\newtheorem{Example}[Thm]{Example}

\newtheorem{Rmk}[Thm]{Remark}

\numberwithin{equation}{section}

\def\frakB{\mathfrak{B}}
\def\frakG{\mathfrak{G}}

\def\bbz{\mathbb{Z}}
\def\bbc{\mathbb{C}}
\def\bbr{\mathbb{R}}
\def\bbz{\mathbb{Z}}

\def\lra{\longrightarrow}

\def\bs{\bar{s}}

\def\aut{\mathrm{Aut}}
\def\aff{\mathrm{Aff}}

\def\Nil{\mathrm{Nil}}

\def\GL{\mathrm{GL}}

\def\e{\epsilon}

\def\Del{\Delta}

\def\boxit#1{\vbox{\hrule\hbox{\vrule\kern3pt
     \vbox{\kern3pt#1\kern3pt}\kern3pt\vrule}\hrule}}

\begin{document}
\title[Topology of iterated $S^1$-bundles]
{Topology of iterated $S^1$-bundles}

\author[Jong Bum Lee and Mikiya Masuda]{Jong Bum Lee and Mikiya Masuda}
\address{Department of Mathematics, Sogang University, Seoul 121-742, KOREA}
\email{jlee@sogang.ac.kr}

\address{Department of Mathematics, Osaka City University, Sumiyoshi-ku,
Osaka 558-8585, JAPAN}
\email{masuda@sci.osaka-cu.ac.jp}

\thanks{The first-named author was partially supported by Mid-career Researcher Program through NRF grant funded by the MEST (No. 2010-0008640) and by the Sogang University Research Grant of 2010. The second-named author was partially supported by Grant-in-Aid for Scientific Research  22540094.}


\date{\today}

\maketitle


\begin{abstract}
In this paper we investigate what kind of manifolds arise as the total spaces of iterated $S^1$-bundles.  A real Bott tower studied in \cite{CMO}, \cite{KM} and \cite{KN} is an example of an iterated $S^1$-bundle.  We show that the total space of an iterated $S^1$-bundle is homeomorphic to an infra-nilmanifold.  A real Bott manifold, which is the total space of a real Bott tower, provides an example of a closed flat Riemannian manifold.  We also show that real Bott manifolds are the only closed flat Riemannian manifolds obtained from iterated $\bbr{P}^1$-bundles.  Finally we classify the homeomorphism types of the total spaces of iterated $S^1$-bundles in dimension 3.
\end{abstract}

\section{Introduction}

In this paper, an $S^1$-bundle is a fiber bundle with the circle $S^1$ as a fiber and an \emph{iterated $S^1$-bundle} of height $n$ is a sequence of  smooth $S^1$-bundles starting with a point:
\begin{equation}  \label{iterated}
X_n\buildrel{}\over\lra X_{n-1}\buildrel{}\over\lra \cdots \buildrel{}\over\lra X_{1}\buildrel{}\over\lra X_{0}=\{\text{a point}\}.
\end{equation}
The total space $X_n$ of an iterated $S^1$-bundle is a closed aspherical manifold of dimension $n$.  Our concern is what kind of aspherical manifolds arise as the total space $X_n$.
If all the $S^1$-bundles $X_i\to X_{i-1}$ in \eqref{iterated} are principal, then one sees that the fundamental group of $X_n$ is nilpotent and hence $X_n$ is homeomorphic to a nilmanifold, and conversely any closed nilmanifold arises as the total space of an iterated \emph{principal} $S^1$-bundle (see \cite[Proposition 3.1]{FaHu}).  Our first main result is the following.

\begin{Thm} \label{main theo1}
The total space $X_n$ of an iterated $S^1$-bundle of height $n$ is homeomorphic to an infra-nilmanifold.  In fact, some $2^{n-1}$-cover of  $X_n$ is homeomorphic to a nilmanifold.
\end{Thm}

The projectivization of a plane bundle, called an $\bbr{P}^1$-bundle, is an $S^1$-bundle, so an iterated $\bbr{P}^1$-bundle is an iterated $S^1$-bundle. The total spaces of iterated $\bbr{P}^1$-bundles are somewhat special.  For instance, the first Betti number $b_1(X_n;\bbz_2)$ of the total space $X_n$ in \eqref{iterated} with $\bbz_2$-coefficient, where $\bbz_2=\bbz/2\bbz$, is at most $n$ and it turns out that $X_n$ is the total space of an iterated $\bbr{P}^1$-bundle if and only if $b_1(X_n;\bbz_2)=n$.

When every plane bundle used to projectivise in the iterated $\bbr{P}^1$-bundle is a Whitney sum of two line bundles, the iterated $\bbr{P}^1$-bundle is called a \emph{real Bott tower} and the total space $X_n$ is called a \emph{real Bott manifold}.  A real Bott manifold provides an example of a flat Riemannian manifold and the diffeomorphism classification of real Bott manifolds has been completed in \cite{CMO}, see also \cite{KM} and \cite{KN}.
Unless every plane bundle used to projectivise in the iterated $\bbr{P}^1$-bundle is a Whitney sum of two line bundles, the total space $X_n$ is not necessarily flat Riemannian.  However, we may ask whether more flat Riemannian manifolds than real Bott manifolds can be produced from iterated $\bbr{P}^1$-bundles. The following theorem says that the answer is no.

\begin{Thm} \label{main theo2}
If the total space of an iterated $\bbr{P}^1$-bundle is homeomorphic to a closed flat Riemannian manifold, then it is homeomorphic to a real Bott manifold.
\end{Thm}

The total space $X_2$ of an iterated $S^1$-bundle of height 2 is either the torus $(S^1)^2$ or the Klein bottle.  However, the total spaces $X_3$ of iterated $S^1$-bundles of height 3 are abundant and we classify them up to homeomorphism.  It turns out that there are six flat Riemannian manifolds, an infinite family of nilmanifolds and an infinite family of infra-nilmanifolds, see Theorem~\ref{3-dim class} for details.  It is known that there are ten homeomorphism classes of closed flat Riemannian manifolds in dimension 3 and our result shows that six of them arise from iterated $S^1$-bundles while four of them arise from iterated $\bbr{P}^1$-bundles.  In a forthcoming paper, we will classify 4-dimensional closed flat Riemannian manifolds arising from iterated $S^1$-bundles.  It is known in \cite{BBNWZ} that there are 74 homeomorphism classes of closed flat Riemannian manifolds in dimension 4 and it turns out that 35 of them arise from iterated $S^1$-bundles while 12 of them arise from iterated $\bbr{P}^1$-bundles.

This paper is organized as follows.  We study fundamental groups of $S^1$-bundles in Section~\ref{sect2} and of iterated $S^1$-bundles in Section~\ref{sect3}.  In Section~\ref{sect4} we prove that the total space of an iterated $S^1$-bundle of height $n$ contains a nilpotent normal subgroup of index $2^{n-1}$ in its fundamental group, which implies Theorem~\ref{main theo1}.   In Section~\ref{sect5} we classify isomorphism classes of possible fundamental groups of the total spaces of iterated $S^1$-bundles of height 3 and then show that those isomorphism classes can be realized by iterated $S^1$-bundles of height 3. Section~\ref{sect6} is devoted to the proof of Theorem~\ref{main theo2}.

\section{$S^1$-bundles} \label{sect2}

When $\xi$ is a plane bundle with an Euclidean fiber metric, the unit circle bundle $S(\xi)$ of $\xi$ is an $S^1$-bundle. Conversely, if the base space $B$ is a closed smooth manifold, then any $S^1$-bundle over $B$ can be regarded as the unit circle bundle of some plane bundle with an Euclidean fiber metric because the inclusion map $\mathrm{O}(2)\to\mathrm{Diff}(S^1)$ is known to be homotopy equivalent so that the structure group of the circle bundle, that is $\mathrm{Diff}(S^1)$, reduces to $\mathrm{O}(2)$.  This also shows that a smooth $S^1$-bundle over a closed smooth manifold is isomorphic to a principal $S^1$-bundle if and only if the $S^1$-bundle is orientable (see \cite[Proposition 6.15]{Mo} for a direct proof).

The projectivization $P(\eta)$ of a plane bundle $\eta$, called the \emph{$\bbr{P}^1$-bundle}, is also an $S^1$-bundle and fiber-wisely double covered by $S(\eta)$.  If $\eta$ is orientable, then $\eta$ admits a complex structure so that one can form its 2-fold tensor product $\eta\otimes_\bbc \eta$ over the complex numbers $\bbc$ and then $P(\eta)=S(\eta\otimes_\bbc \eta)$.

\begin{Lemma}\label{pi}
Let $S^1\buildrel{i}\over\to X\buildrel\pi\over\to B$ be an $S^1$-bundle over an arcwise connected space $B$ and let $\pi_1(B)$ be finitely presented as follows
$$
\langle x_1,\cdots,x_p\ |\ f_j(x_1,\cdots,x_p)=1\ (1\le j\le q)\rangle
$$
{and let $i_*:\pi_1(S^1)\to\pi_1(X)$ be injective}. Then $\pi_1(X)$ has a presentation of the form
$$
\langle x_1,\cdots,x_p,\sigma\ |\ x_i{\sigma}x_i^{-1}={\sigma}^{\pm1}, f_j(x_1,\cdots,x_p)=\sigma^{a_j}\ (1\le i\le p,1\le j\le q)\rangle
$$
for some integers $a_j$.

Moreover, the following are equivalent:
\begin{enumerate}
\item[$(1)$] the $S^1$-bundle $X\to B$ is fiber-wisely double covered by another $S^1$-bundle,
\item[$(2)$] every integer $a_j$ above is even,
\item[$(3)$] $b_1(X,\bbz_2)=b_1(B,\bbz_2)+1$ where $b_1(\ ,\bbz_2)$ denotes the first Betti number with $\bbz_2$-coefficient.
\end{enumerate}
\end{Lemma}

\begin{proof}
The $S^1$-bundle $X\to B$ induces a short exact sequence $1\to\pi_1(S^1)\to\pi_1(X)\to\pi_1(B)\to1$. Taking $\sigma$ as a generator of $\pi_1(S^1)$, it can be seen easily that the first part of the Lemma holds.

(1) $\Longrightarrow$ (2).
Assume that the $S^1$-bundle $X\to B$ is fiber-wisely double covered by another $S^1$-bundle $Y\to B$. Then there is a fiber preserving map $\phi$ between them
$$
\CD
S^1@>{\bar\phi}>>S^1\\
@VV{i'}V@VV{i}V\\
Y@>{\phi}>>X\\
@VV{\pi'}V@VV{\pi}V\\
B@>{=}>>B
\endCD
$$
where $\phi:Y\to X$ and the restriction on the fiber $\bar\phi:S^1\to S^1$ are double covering projections. Therefore $\phi$ induces the following commuting diagram between exact sequences of groups
$$
\CD
\cdots@>>>\pi_1(S^1)@>{i'_*}>>\pi_1(Y)@>{\pi'_*}>>\pi_1(B)@>>>1\\
@.@VV{\bar\phi_*}V@VV{\phi_*}V@VV{=}V\\
\cdots@>>>\pi_1(S^1)@>{i_*}>>\pi_1(X)@>{\pi_*}>>\pi_1(B)@>>>1
\endCD
$$
Since $i_*$ and $\bar\phi_*$ are injective, $i_*\bar\phi_*=\phi_*i'_*$ yields that $i'_*$ is injective. Let $\tau$ be a generator of $\pi_1(S^1)$ of the fiber $S^1$ of the bundle $Y\to B$ so that $\bar\phi_*(\tau)=\sigma^2$. 
We choose a lift of $x_i\in \pi_1(B)$ to $\pi_1(X)$ through $\pi_*$ for each $i$ and use the same notation $x_i$ for the lift. Then, we choose a lift of $x_i\in \pi_1(X)$ to $\pi_1(Y)$ through $\phi_*$ and denote it by $y_i$.  Recalling that
\begin{align*}
\pi_1(Y)&=\langle y_1,\cdots,y_p,\tau\mid y_i{\tau}y_i^{-1}={\tau}^{\pm1}, f_j(y_1,\cdots,y_p)=\tau^{b_j}\rangle,\\
\pi_1(X)&=\langle x_1,\cdots,x_p,\sigma\mid x_i{\sigma}x_i^{-1}={\sigma}^{\pm1}, f_j(x_1,\cdots,x_p)=\sigma^{a_j}\rangle
\end{align*}
for some integers $a_i$ and $b_j$, we must have that
$\phi_*(f_j(y_1,\cdots,y_p))=\phi_*(\tau^{b_j})$ and so $f_j(x_1,\cdots,x_p)=\sigma^{2b_j}$.
Hence $a_j=2b_j$ for all $j$.

(2) $\Longrightarrow$ (1).
Conversely suppose that the fundamental group of the total space $X$ of the $S^1$-bundle $X\to B$ has a presentation of the form as above with all the integers $a_j$ even. Consider the subgroup $H$ of $\pi_1(X)$ generated by $x_1,\cdots,x_p$ and $\sigma^2$. Then $H$ has index $2$ in $\pi_1(X)$. Let $Y$ be the double covering space of $X$ corresponding to $H$ with covering projection $\phi:Y\to X$. Then $\pi'=\pi\phi:Y\to B$ is a bundle with fiber $F=\phi^{-1}(S^1)$ and we have the commutative diagram
$$
\CD
F@>{\bar\phi}>>S^1\\
@VV{i'}V@VV{i}V\\
Y@>{\phi}>>X\\
@VV{\pi'}V@VV{\pi}V\\
B@>{=}>>B
\endCD
$$
where $\bar\phi$ is the restriction of $\phi$.

Now we will show that $F=S^1$ and $\bar\phi$ is a double covering projection. Notice that $\phi_*:\pi_1(Y)=H\to\pi_1(X)$ is the inclusion $H\hookrightarrow \pi_1(X)$ and hence the composition $\pi'_*=\pi_*\phi_*:\pi_1(Y)\to\pi_1(B)$ is surjective by the choice of $H$. It follows that $\pi_0(F)=1$, i.e., $F$ is arcwise connected. Hence $\bar\phi:F\to S^1$ is a (double) covering projection \cite[Lemma~2.1, p.\,150]{M} and so $F=S^1$.

(2) $\Leftrightarrow$ (3).  We note that $H_1(X;\bbz_2)=H_1(X;\bbz)\otimes \bbz_2$ which follows from the universal coefficient theorem for homology groups because $H_0(X;\bbz)$ is torsion free.  Therefore, $H_1(X;\bbz_2)$ agrees with the abelianization of $\pi_1(X)$ over $\bbz_2$. Looking at the description of $\pi_1(X)$ and $\pi_1(B)$, one easily sees the equivalence between (2) and (3).
\end{proof}

The last part of the proof above essentially proves that $b_1(X;\bbz_2)=b_1(B;\bbz_2)$ or $b_1(B;\bbz_2)+1$.  This fact can also be seen in terms of 2nd Stiefel-Whitney class as follows.

\begin{Lemma} \label{first betti}
Let $p\colon X\to B$ be the unit circle bundle of a plane bundle $\xi$ over a connected space $B$.  Then $b_1(X;\bbz_2)=b_1(B;\bbz_2)$ or $b_1(B;\bbz_2)+1$ and the former occurs when $w_2(\xi)\not= 0$ and the latter occurs when $w_2(\xi)=0$.
\end{Lemma}

\begin{proof}
Consider the Thom-Gysin sequence associated with the $S^1$-bundle $p\colon X\to B$:
\[
0=H^{-1}(B)\stackrel{\cup w_2(\xi)}\longrightarrow H^1(B)\stackrel{p^*}\longrightarrow H^1(X) \longrightarrow H^0(B)\stackrel{\cup w_2(\xi)}\longrightarrow H^2(B)
\]
where the coefficients are taken with $\bbz_2$.  Here the last map above is injective if $w_2(\xi)\not=0$ and zero if $w_2(\xi)=0$. This implies the lemma since $H^0(B)\cong \bbz_2$.
\end{proof}

\begin{Rmk} \label{zp}
The proof of the equivalence between (2) and (3) in Lemma~\ref{pi} works with $\bbz_p$ coefficient for any prime numer $p$, and we have that $b_1(X;\bbz_p)\le b_1(B;\bbz_p)+1$ and the equality holds if and only if every integer $a_j$is divisible by $p$, where $b_1(\ ;\bbz_p)$ denotes the first Betti number with $\bbz_p$ coefficient.  Therefore, the $S^1$-bundle $\pi\colon X\to B$ is trivial if and only if $b_1(X;\bbz_p)=b_1(B;\bbz_p)+1$ for any prime number $p$.  
\end{Rmk}

\section{Iterated $S^1$-bundles} \label{sect3}

An \emph{iterated $S^1$-bundle} of height $n$ is a sequence of smooth $S^1$-bundles starting with a point:
\begin{align}\label{tower of PS}
X_n\buildrel{}\over\lra X_{n-1}\buildrel{}\over\lra \cdots
\buildrel{}\over\lra X_{1}\buildrel{}\over\lra X_{0}=\{\text{a point}\}.
\end{align}
Each $X_i$ is a closed connected aspherical manifold of dimension $i$ for $i=1,2,\dots,n$ and the $S^1$-bundle $X_i\to X_{i-1}$ induces a short exact sequence:
\[
1\to \pi_1(S^1)\to \pi_1(X_i)\to \pi_1(X_{i-1})\to1.
\]
The total space $X_n$ is diffeomorphic to an $n$-dimensional torus if every $S^1$-bundle $X_i\to X_{i-1}$ in \eqref{tower of PS} is trivial. 
The converse is also true.  

\begin{Prop}
The following are equivalent:
\begin{enumerate}
\item $X_n$ is diffeomorphic to an $n$-dimensional torus,
\item $H_1(X_n;\bbz)$ is isomorphic to $\bbz^n$,
\item every $S^1$-bundle $X_i\to X_{i-1}$ in \eqref{tower of PS} is trivial.
\end{enumerate}
\end{Prop}
\begin{proof}
It suffices to prove (2) $\Rightarrow$ (3).  Suppose that $H_1(X_n;\bbz)$ is isomorphic to $\bbz^n$.  Then $b_1(X_n;\bbz_p)=n$ for any prime number $p$ by the universal coefficient theorem.  On the other hand, repeated use of Remark~\ref{zp} shows that $b_1(X_n;\bbz_p)\le n$ and the equality holds for any prime number $p$ only when every $S^1$-bundle $X_i\to X_{i-1}$ in \eqref{tower of PS} is trivial.  
\end{proof}

Unless the $S^1$-bundles in \eqref{tower of PS} are trivial, the topology of $X_n$ is complicated in general.

\begin{Lemma}\label{presentation}
$\pi_1(X_n)$ has a presentation of the form
\begin{align*}
\left\langle s_1,\cdots, s_n\ \Big|\
s_is_js_i^{-1}=s_n^{a^n_{i,j}}\cdots s_{j+1}^{a^{j+1}_{i,j}}s_j^{\e_{ij}} \ (1\le i<j\le n)
\right\rangle
\end{align*}
where $\e_{ij}=\pm 1$ and the $a^k_{i,j}$ are some integers.
Moreover, the $S^1$-bundle $X_j\to X_{j-1}$ in \eqref{tower of PS} is orientable $($equivalently, principal$)$ if and only if $\e_{ij}=1$ for all $i=1,\cdots,j-1$.
\end{Lemma}

\begin{proof}
The former statement follows by applying Lemma~\ref{pi} inductively.

The proof of the latter is as follows.
Note that the $S^1$-bundle $X_j\to X_{j-1}$ is orientable if and only if any loop in the base space $X_{j-1}$ induces the identity map on the first cohomology group of the fiber over the base point of the loop. Equivalently, $\pi_1(X_{j-1},b)$ acts on $H^1(F_b)$ trivially for all $b\in B$. This exactly means that $s_i{s_j}s_i^{-1}={s_j}$ for all the generators $s_i$ of $\pi_1(X_{j-1})$.
\end{proof}

Since the projectivization of a plane bundle, that is an $\bbr{P}^1$-bundle, is an $S^1$-bundle, an iterated $\bbr{P}^1$-bundle is an iterated $S^1$-bundle.

\begin{Prop} \label{char of RP}
Let $X_n$ be the total space of an iterated $S^1$-bundle \eqref{tower of PS}.  Then $b_1(X_n;\bbz_2)\leq n$.  Moreover, the following are equivalent.
\begin{enumerate}
\item[$(1)$] $X_n$ is the total space of an iterated $\bbr P^1$-bundle.
\item[$(2)$] All the exponents $a_{i,j}^k$ in Lemma~\ref{presentation} are even.
\item[$(3)$] $b_1(X_n;\bbz_2)=n$.
\end{enumerate}
\end{Prop}

\begin{proof}
The former statement follows from Lemma~\ref{first betti}.  The latter follows by applying Lemma~\ref{pi} repeatedly.
\end{proof}

Finally, we give an example of iterated $\bbr P^1$-bundle which motivated the study of this paper.

\begin{Example}[Real Bott tower] \label{real Bott}
An iterated $\bbr P^1$-bundle of height $n$:
\begin{align} \label{real Bott tower}
B_n\buildrel{}\over\lra B_{n-1}\buildrel{}\over\lra \cdots
\buildrel{}\over\lra B_{1}\buildrel{}\over\lra B_{0}=\{\text{a point}\},
\end{align}
where each fibration $B_i\to B_{i-1}$ for $i=1,2,\dots,n$ is the projectivization of a Whitney sum of two line bundles over $B_{i-1}$ is called a \emph{real Bott tower}  of height $n$, and the total space $B_n$ is called a \emph{real Bott manifold}.  At each stage, one of the two line bundles may be assumed to be trivial without loss of generality because projectivization remain unchanged under tensor product with a line bundle. The same construction works in the complex category and in this case the tower is called a Bott tower and the total space $B_n$ is called a Bott manifold. A two stage Bott manifold is nothing but a Hirzebruch surface.  A Bott manifold provides an example of a closed smooth toric variety and a real Bott manifold provides an example of a closed smooth real toric variety.

A real Bott manifold $B_n$ also provides an example of a flat Riemannian manifold.  In fact, it can be described as the quotient of $\bbr^n$ by a group $\pi_n$ generated by Euclidean motions $s_i$'s $(i=1,\dots,n)$ on $\bbr^n$ defined by
\[
s_i(x_1,\dots,x_n):=(x_1,\dots,x_{i-1}, x_i+\frac{1}{2}, \epsilon_{i+1}^ix_{i+1},\dots, \epsilon_{n}^ix_n),
\]
where $\epsilon_j^i=\pm 1$ for $1\le i<j\le n$ and $\epsilon^i_j$'s are determined by the line bundles used to define the real Bott tower \eqref{real Bott tower}.  The action of $\pi_n$ on $\bbr^n$ is free so that $\pi_n$ is the fundamental group of the real Bott manifold $B_n$.  It is generated by $s_i$'s $(i=1,\dots,n)$ with relations
\[
s_is_js_i^{-1}=s_j^{\e^i_j} \quad\text{for $1\le i<j\le n$}.
\]
The subgroup of $\pi_n$ generated by $s_i^2$'s $(i=1,\dots,n)$ is the translations $\bbz^n$ and the quotient $\pi_n/\bbz^n$ is an elementary 2-group of rank $n$. Note that the natural projections $\bbr^n\to \bbr^{n-1}\to \dots\to \bbr^1\to \bbr^0=\{\text{a point}\}$ induce a real Bott tower.

The diffeomorphism classification of real Bott manifolds has been completed in \cite{CMO}.  The paper \cite{CMO} also relates the diffeomorphism classification of real Bott manifolds with the classification of acyclic digraphs (directed graphs with no direct cycles) up to some equivalence.
\end{Example}

\section{Infra-nilmanifolds} \label{sect4}

The purpose of this section is to prove Theorem~\ref{main theo1} in the Introduction.  We continue to use notations in Section~\ref{sect3}.
A group $G$ is called \emph{supersolvable} if there exists a finite normal series
$$
G=\bar{G}_1\supset \bar{G}_2\supset \cdots\supset \bar{G}_c\supset \bar{G}_{c+1}=1
$$
such that each quotient group $\bar{G}_i/\bar{G}_{i+1}$ is cyclic and each $\bar{G}_i$ is normal in $G$.

\begin{Lemma}\label{supersolv}
$\pi_1(X_n)$ is a supersolvable group.
\end{Lemma}

\begin{proof}
We consider the subgroups $\bar\pi_j$ of $\pi_1(X_n)$ generated by $s_{j},\cdots, s_n$. Then we have a finite normal series
$$
\pi_1(X_n)=\bar\pi_1\supset\bar\pi_{2}\supset\cdots\supset \bar\pi_n\supset\bar\pi_{n+1}=1
$$
such that $\bar\pi_i/\bar\pi_{i+1}\cong\langle s_i\rangle$. By Lemma~\ref{presentation}, it follows easily that each $\bar\pi_i$ is a normal subgroup of $\pi_1(X_n)$.
\end{proof}

However, the normal series in the above proof is not always a central series. This implies that $\pi_1(X_n)$ is not always a nilpotent group. We will show in Theorem~\ref{virt nilp} that $\pi_1(X_n)$ is always virtually nilpotent. Note further that the subgroup $\bar\pi_i$ of $\pi_1(X_n)$ is isomorphic to $\pi_1(X_{n-i+1})$.

The projection $\pi_1(X_i)\to\pi_1(X_{i-1})$ sends $s_j$ to $s_j$ for $j=1,\cdots,i-1$ with kernel $\pi_1(S^1)=\langle s_i\rangle$.
For the simplicity, we will write $\pi_i=\pi_1(X_i)$ and $A_i=\pi_1(S^1)$ with generator $s_i$. Thus we have a short exact sequence
\begin{align}\label{SES}
1\lra A_i\lra\pi_i\lra\pi_{i-1}\lra1.
\end{align}
Let $\Gamma_{i}$ be the subgroup of $\pi_i$ generated by $s_1^2,\cdots, s_i^2$. Then $\Gamma_i$ is mapped onto $\Gamma_{i-1}$ under the projection $\pi_i\to\pi_{i-1}$ with kernel $\langle s_i^2\rangle$, which induces a short exact sequence
\begin{align}\label{SES1}
1\lra\langle s_i^2\rangle\lra\Gamma_i\lra\Gamma_{i-1}\lra1.
\end{align}

\begin{Lemma}\label{normal}
$\Gamma_n$ is a normal subgroup of $\pi_n$ with index $2^n$.
\end{Lemma}

\begin{proof}
For each $s_i$, we denote by $c(s_i)$ the conjugation by $s_i$. Since $c(s_i)(s_n)=s_n^{\e_{in}}$, the conjugate automorphism $c(s_i)$ on $\pi_n$ induces the following commutative diagram
$$
\CD
1@>>>A_n@>>>\pi_n@>>>\pi_{n-1}@>>>1\\
@.@VV{c(s_i)}V@VV{c(s_i)}V@VV{c(\bar{s}_i)}V\\
1@>>>A_n@>>>\pi_n@>>>\pi_{n-1}@>>>1
\endCD
$$
where $\bar{s}_i$ is the image of $s_i$ under $\pi_n\to\pi_{n-1}$.
This diagram gives rise to the following commutative diagram of short exact sequences
$$
\CD
1@>>>\langle s_n^2\rangle@>>>\Gamma_n@>>>\Gamma_{n-1}@>>>1\\
@.@VV{c(s_i)}V@VV{c(s_i)}V@VV{c(\bar{s}_i)}V\\
1@>>>\langle s_n^2\rangle@>>>\Gamma'_n@>>>\Gamma'_{n-1}@>>>1
\endCD
$$
where $\Gamma'_n$ and $\Gamma'_{n-1}$ are the images of $\Gamma_n$ and $\Gamma_{n-1}$ under $c(s_i)$ and $c(\bar{s}_i)$ respectively.
In order to show that $\Gamma_n$ is a normal subgroup of $\pi_n$, it suffices to show that $\Gamma'_n=\Gamma_n$. For this purpose we will use induction on $n$. It is clear that $\Gamma'_1=\Gamma_1$. Assume that $\Gamma'_{n-1}=\Gamma_{n-1}$. Consider an element $s_is_j^2s_i^{-1}$ of $\Gamma'_n$. It is mapped to the element $\bar{s}_i\bar{s}_j^2\bar{s}_i^{-1}$ of $\Gamma'_{n-1}=\Gamma_{n-1}$. Hence $\bar{s}_i\bar{s}_j^2\bar{s}_i^{-1}$ is a word of $\bar{s}_1^2,\cdots, \bar{s}_{n-1}^2$. This therefore implies that $s_is_j^2s_i^{-1}$ is a word of ${s}_1^2,\cdots, {s}_{n-1}^2,s_n^2$, which means that $s_is_j^2s_i^{-1}\in\Gamma_n$. Consequently $\Gamma'_n=\Gamma_n$.

Furthermore, we have the following commutative diagram of short exact sequences
$$
\CD
@.1@.1@.1@.\\
@.@AAA@AAA@AAA\\
1@>>>\bbz_2@>>>\pi_n/\Gamma_n@>>>\pi_{n-1}/\Gamma_{n-1}@>>>1\\
@.@AAA@AAA@AAA\\
1@>>>A_n@>>>\pi_n@>>>\pi_{n-1}@>>>1\\
@.@AAA@AAA@AAA\\
1@>>>\langle s_n^2\rangle@>>>\Gamma_n@>>>\Gamma_{n-1}@>>>1\\
@.@AAA@AAA@AAA\\
@.1@.1@.1
\endCD
$$
This, in particular, shows that the order of $\pi_n/\Gamma_n$ equals $2^n$ by induction.
\end{proof}

\begin{Lemma}\label{virt nilp}
$\Gamma_n$ is a nilpotent group of rank $n$. Therefore, $\pi_n$ is a torsion-free virtually nilpotent group of rank $n$.
\end{Lemma}

\begin{proof}
It suffices to show that $\Gamma_n$ has a finite central series
$$
\Gamma_n=\Gamma^1\supset\Gamma^2\supset \cdots\supset\Gamma^c\supset\Gamma^{c+1}=1
$$
such that the quotient groups $\Gamma^i/\Gamma^{i+1}$ are isomorphic to some $\bbz^{k_i}$. We will use induction on $n$ to show that the series
\begin{align}\label{series}
\Gamma_{n}=\langle s_1^2,\cdots,s_{n}^2\rangle\supset \langle s_2^2,\cdots, s_{n}^2\rangle\supset\cdots\supset\langle s_{n-1}^2,s_n^2\rangle\supset\langle s_n^2\rangle\supset\{1\}
\end{align}
is a required central series with successive quotient groups isomorphic to $\bbz$.

The case where $n=1$ is obvious and hence we assume the following: $\Gamma_{n-1}$ has such a central series. To avoid confusion let us use $\bs_1,\cdots,\bs_{n-1}$ in the presentation of $\pi_{n-1}$ given in Lemma~\ref{presentation} so that $s_i\in\pi_n$ is mapped to $\bs_i\in\pi_{n-1}$ for $i=1,\cdots,n-1$. Then $\Gamma_{n-1}=\langle \bs_1^2,\cdots, \bs_{n-1}^2\rangle$ with index $2^{n-1}$ in $\pi_{n-1}$, and by induction hypothesis, $\Gamma_{n-1}$ has a central series
\begin{align}\label{b-series}
\Gamma_{n-1}=\langle \bs_1^2,\cdots,\bs_{n-1}^2\rangle\supset \langle \bs_2^2,\cdots, \bs_{n-1}^2\rangle\supset\cdots\supset\langle \bs_{n-1}^2\rangle\supset\{1\}
\end{align}
with successive quotient groups isomorphic to $\bbz$. Using the short exact sequence $1\to\langle s_n^2\rangle\to\Gamma_n\buildrel{p}\over\to\Gamma_{n-1}\to1$, we take the pullback of the series (\ref{b-series}). Namely, for each subgroup $\langle \bs_i^2,\cdots, \bs_{n-1}^2\rangle$ of $\Gamma_{n-1}$, we consider the subgroup $p^{-1}(\langle \bs_i^2,\cdots, \bs_{n-1}^2\rangle)$ of $\Gamma_n$. This group fits in a short exact sequence $1\to\langle s_n^2\rangle\to p^{-1}(\langle \bs_i^2,\cdots, \bs_{n-1}^2\rangle)\to\langle \bs_i^2,\cdots, \bs_{n-1}^2\rangle\to1$, which induces that $p^{-1}(\langle \bs_i^2,\cdots, \bs_{n-1}^2\rangle)=\langle s_i^2,\cdots, s_{n}^2\rangle$. Therefore, we have the following commutative diagram
$$
\CD
1@>>> \langle s_n^2\rangle @>>>\Gamma_n@>>>\Gamma_{n-1}@>>>1\\
@.@AA=A@AAA@AAA\\
1@>>> \langle s_n^2\rangle @>>>\langle s_2^2,\cdots, s_{n}^2\rangle@>>>
\langle \bs_2^2,\cdots, \bs_{n-1}^2\rangle@>>>1\\
@.@AA=A@AAA@AAA\\
@.\vdots@.\vdots@.\vdots\\
@.@AA=A@AAA@AAA\\
1@>>>\langle s_n^2\rangle@>>> \langle s_{n-2}^2,s_{n}^2\rangle@>>>\langle \bs_{n-1}^2\rangle@>>>1\\
@.@AA=A@AAA@AAA\\
1@>>>\langle s_n^2\rangle@>>> \langle s_{n}^2\rangle@>>>1@>>>1
\endCD
$$
Finally we note that since the most right vertical is a central series, so is the induced middle vertical. Clearly the rank of $\Gamma_n$ is $n$.
\end{proof}

In fact, $\pi_n$ contains another nilpotent normal subgroup which is slightly larger than $\Gamma_n$ as is shown in the following lemma.

\begin{Lemma} \label{Lambda_n}
Let $\Lambda_n$ be the subgroup of $\pi_n$ generated by $s_1^2,\cdots, s_{n-1}^2,s_n$. Then $\Lambda_n$ is a nilpotent normal subgroup of $\pi_n$ which has rank $n$ and index $2^{n-1}$.
\end{Lemma}

\begin{proof}
Under the short exact sequence $1\to\langle s_n\rangle\to\pi_n\to\pi_{n-1}\to1$, we take the pullback of the subgroup $\Gamma_{n-1}$ of $\pi_{n-1}$. Then we obtain the short exact sequence  $1\to\langle s_n\rangle\to\Lambda_n\to\Gamma_{n-1}\to1$. Since $\Gamma_{n-1}$ is normal in $\pi_{n-1}$, it follows that $\Lambda_n$ is a normal subgroup of $\pi_n$.

On the other hand, $\Lambda_n$ fits in the following short exact sequence $1\to\Gamma_n\to\Lambda_n\to\bbz_2\to1$. Since $s_is_ns_i^{-1}=s_n^{\e_{in}}$, we have $s_i^2s_ns_i^{-2}=s_n$ and so the extension is central. Hence since $\Gamma_n$ is nilpotent, we see that $\Lambda_n$ is nilpotent.
\end{proof}

Now we are in a position to prove our first main theorem stated in the Introduction.

\begin{Thm}\label{infra-nil}
The total space $X_n$ of an iterated $S^1$-bundle of height $n$ is homeomorphic to an infra-nilmanifold.  In fact, some $2^{n-1}$-cover of $X_n$ is homeomorphic to a nilmanifold.
\end{Thm}

\begin{proof}
Let $X_n$ be the total space of an iterated $S^1$-bundle and let $\pi_n$ be its fundamental group as before.
By \cite[Corollary~3.2.1]{KLR}, there is an infra-nilmanifold $X$ whose fundamental group is isomorphic to $\pi_n$. Therefore, two aspherical manifolds $X_n$ and $X$ are homotopic. By \cite[Theorem~6.3]{FH}, $X_n$ and $X$ are homeomorphic except possibly for $n=3,4$.

Since $X_4$ is aspherical and $\pi_4$ is virtually nilpotent, $X_4$ has an infra-nil structure by \cite[Corollary~2.21]{FJ}. (In fact, this is true for all $n\ne3$. See also F. Quinn's Math Review of the paper \cite{FH}.) Namely, $X_4$ is homeomorphic to an infra-nilmanifold.

It is well known that all $3$-dimensional infra-nilmanifolds are Seifert manifolds. It is evident that the Seifert manifolds $X_3$ and $X$ are sufficiently large, see \cite[Proposition~2]{H73}. By works of Waldhausen \cite{Wald} and Heil \cite[Theorem~A]{H69}, $X_3$ is homeomorphic to $X$.

By Lemmas~\ref{normal} and~\ref{Lambda_n}, $\pi_n$ has a normal nilpotent subgroup $\Lambda_n$ of index $2^{n-1}$. The covering space associated with the nilpotent group $\Lambda_n$ is a $2^{n-1}$-cover of $X_n$ and it is homeomorphic to a nilmanifold.
\end{proof}

\begin{Rmk}
The closed nilmanifolds are precisely the total spaces of iterated \emph{principal} $S^1$-bundles up to homeomorphism as remarked in the Introduction.
\end{Rmk}

We conclude this section with the following lemma.

\begin{Lemma}\label{Bieberbach}
$\pi_n$ is isomorphic to a Bieberbach group $($in other words, $X_n$ is homeomorphic to a flat Riemannian manifold$)$ if and only if $\Gamma_n$ is isomorphic to $\bbz^n$.
\end{Lemma}

\begin{proof}
The if part is clear.  Suppose that $\pi_n$ is a Bieberbach group.  Then $\bbr^n/\pi_n$ is a flat Riemannian manifold, so is its finite cover $\bbr^n/\Gamma_n$.  On the other hand, it is known by Gromoll-Wolf \cite{GW} and Yau \cite{Y} that if the fundamental group of a compact nonpositively curved manifold is nilpotent, then it is abelian.  Therefore, $\Gamma_n$ is isomorphic to $\bbz^n$.
\end{proof}

\section{Iterated $S^1$-bundles of height $3$}\label{sect5}

In this section we classify the 3-dimensional total spaces obtained as iterated $S^1$-bundles of height $3$ up to homeomorphism (equivalently up to diffeomorphism because diffeomorphism classification is the same as homeomorphism classification in dimension 3).  This classification reduces to the classification of isomorphism classes of their fundamental groups by Theorem~\ref{infra-nil}.

\subsection{Isomorphism classes of $\pi_3$}
In the 3-dimensional case, by Lemma~\ref{presentation}, the fundamental group $\pi_3$ of the total space of an iterated $S^1$-bundle of height $3$ is generated by $s_1,s_2, s_3$ with relations
\begin{equation} \label{-1}
s_{1}s_{2}s_{1}^{-1}=s_3^{a}s_{2}^{\e}, \quad s_{1}s_3s_{1}^{-1}=s_3^{\e_{1}},\quad s_{2}s_3s_{2}^{-1}=s_3^{\e_{2}}
\end{equation}
where $a\in\bbz$ and $\e, \e_1,\e_2\in \{\pm 1\}$.  We shall denote the group $\pi_3$ with the relation \eqref{-1} by $\Pi(a,\e,\e_1,\e_2)$.

\begin{Lemma} \label{3-flat}
$\Pi(a,\e,\e_1,\e_2)$ is a Bieberbach group if and only if $(\e+\e_1)(\e_2+1)a=0$.
\end{Lemma}

\begin{proof}
By Lemma~\ref{Bieberbach}, $\Pi(a,\e,\e_1,\e_2)$ is a Bieberbach group if and only if $s_i^2s_j^2=s_j^2s_i^2$ for $1\le i<j\le 3$.  The latter two identities in \eqref{-1} imply that $s_3$ commutes with $s_1^2$ and $s_2^2$.  Therefore it suffices to show that $s_1^2s_2^2=s_2^2s_1^2$ if and only if $(\e+\e_1)(\e_2+1)a=0$.
We note that the latter two identities in \eqref{-1} imply
\begin{equation} \label{2}
s_is_3^b=s_3^{\e_i b}s_i \quad\text{for $i=1,2$ and $b\in \bbz$}.
\end{equation}
We distinguish two cases according to the value of $\e$.

The case where $\e=1$.  In this case $s_1s_2=s_3^as_2s_1$ by the first identity in \eqref{-1}. Using this together with \eqref{2}, we have
\[
\begin{split}
s_1^2s_2^2&=s_1(s_1s_2)s_2=s_1(s_3^as_2s_1)s_2=s_3^{\e_1 a}(s_1s_2)(s_1s_2)\\
&=s_3^{\e_1 a}(s_3^as_2s_1)(s_3^as_2s_1)=s_3^{\e_1 a+a+\e_1\e_2 a}s_2(s_1s_2)s_1\\
&=s_3^{\e_1 a+a+\e_1\e_2 a}s_2(s_3^as_2s_1)s_1=s_3^{\e_1 a+a+\e_1\e_2 a+\e_2 a}s_2^2s_1^2.
\end{split}
\]
Therefore $s_1^2s_2^2=s_2^2s_1^2$ if and only if the exponent of $s_3$ in the last term above is zero.  This is equivalent to the assertion in the lemma because $\e=1$.

The case where $\e=-1$.  In this case $s_1s_2=s_3^as_2^{-1}s_1$ by the first identity in \eqref{-1}. Moreover, by taking inverse at the both sides of the first identity in \eqref{-1} and using \eqref{2}, we obtain $s_1s_2^{-1}=s_3^{-\e_2 a}s_2s_1$.  Using these two identities together with \eqref{2}, we have
\[
\begin{split}
s_1^2s_2^2&=s_1(s_1s_2)s_1=s_1(s_3^as_2^{-1}s_1)s_2=s_3^{\e_1 a}(s_1s_2^{-1})(s_1s_2)\\
&=s_3^{\e_1 a}(s_3^{-\e_2 a}s_2s_1)(s_3^as_2^{-1}s_1)=s_3^{\e_1 a-\e_2 a+\e_1\e_2 a}s_2(s_1s_2^{-1})s_1\\
&=s_3^{\e_1 a-\e_2 a+\e_1\e_2 a}s_2(s_3^{-\e_2 a}s_2s_1)s_1=s_3^{\e_1 a-\e_2 a+\e_1\e_2 a-a}s_2^2s_1^2.
\end{split}
\]
Therefore $s_1^2s_2^2=s_2^2s_1^2$ if and only if the exponent of $s_3$ in the last term above is zero.  This is equivalent to the assertion in the lemma because $\e=-1$.
\end{proof}

Lemma~\ref{3-flat} implies that when $(\e,\e_1,\e_2)=(1,1,1)$ or $(-1,-1,1)$, $\Pi(a,\e,\e_1,\e_2)$ is a Bieberbach group if and only if $a=0$.
This condition that $(\e,\e_1,\e_2)=(1,1,1)$ or $(-1,-1,1)$ appears from another viewpoint as is seen in (the proof of) the following lemma.

\begin{Lemma} \label{lifting}
Unless $(\e,\e_1,\e_2)=(1,1,1)$ or $(-1,-1,1)$,
\[
\Pi(a,\e,\e_1,\e_2)\cong\begin{cases} \Pi(0,\e,\e_1,\e_2) &\quad\text{if $a$ is even},\\
\Pi(1,\e,\e_1,\e_2) &\quad\text{if $a$ is odd}.\end{cases}
\]
\end{Lemma}

\begin{proof}
Changing the lift of $s_1$ and $s_2$, we may replace
$$
s_1\mapsto s_3^{-b}s_1,\ s_2\mapsto s_3^{-c}s_2,\ s_3\mapsto s_3
$$
where $b$ and $c$ can be any integers.  Setting
\[
t_1=s_3^{-b}s_1,\ t_2=s_3^{-c}s_2,\ t_3=s_3,
\]
the second and third identities of \eqref{-1} remain unchanged with $s$ replaced by $t$ but the first one turns into
\begin{equation} \label{0}
(t_3^bt_1)(t_3^ct_2)(t_3^bt_1)^{-1}=t_3^a(t_3^ct_2)^\e.
\end{equation}
The left hand side of \eqref{0} reduces to
\[
t_3^{b+\e_1 c-\e_2 b}t_1t_2t_1^{-1}
\]
while the right hand side of \eqref{0} reduces to
\[
\begin{cases}
t_3^{a+c}t_2 \quad&\text{when $\e=1$},\\
t_3^{a-\e_2 c}t_2^{-1} \quad&\text{when $\e=-1$}.
\end{cases}
\]
Therefore the first identity in \eqref{-1} turns into
\[
t_1t_2t_1^{-1}=\begin{cases} t_3^{a+(\e_2-1)b-(\e_1-1)c}t_2 \quad&\text{when $\e=1$},\\
t_3^{a+(\e_2-1)b-(\e_1+\e_2)c}t_2^{-1} \quad&\text{when $\e=-1$}.
\end{cases}
\]
This implies the lemma.
\end{proof}

There are more isomorphisms among groups $\Pi(a,\e,\e_1,\e_2)$.

\begin{Lemma} \label{4 iso}
The following isomorphisms hold:
\begin{enumerate}
\item[$(1)$] $\Pi(a,\e,\e_1,\e_2)\cong \Pi(-a,\e,\e_1,\e_2)$.
\item[$(2)$] $\Pi(a,\e,\e_1,\e_2)\cong\Pi(a,\e,\e_1\e_2,\e_2)$.
\item[$(3)$] $\Pi(a,1,\e_1,\e_2)\cong\Pi(a,1,\e_2,\e_1)$.
\item[$(4)$] $\Pi(a,\e,-\e,1)\cong\Pi(a,-\e,\e,1)$.
\end{enumerate}
\end{Lemma}

\begin{proof}
The following isomorphisms are desired ones for the first three cases:

(1)  $s_1\to s_1,\ s_2\to s_2,\ s_3\to s_3^{-1}$.

(2)  $s_1\to s_1s_2,\ s_2\to s_2,\ s_3\to s_3$.

(3)  $s_1\to s_2,\ s_2\to s_1,\ s_3\to s_3^{-1}$.

\noindent
It would be obvious that the first two above are the desired isomorphisms.  We shall check it for (3).  We set $t_1=s_2$, $t_2=s_1$ and $t_3=s_3^{-1}$.  Then
\[
\begin{split}
t_1t_2t_1^{-1}&=s_2s_1s_2^{-1}=s_3^{-a}s_1=t_3^at_2,\\
t_1t_3t_1^{-1}&=s_2s_3^{-1}s_2^{-1}=s_3^{-\e_2}=t_3^{\e_2},\\
t_2t_3t_2^{-1}&=s_1s_3^{-1}s_1^{-1}=s_3^{-\e_1}=t_3^{\e_1},
\end{split}
\]
and this proves the isomorphism (3) in the lemma.

The proof of (4) is as follows.  By Lemma~\ref{lifting} we may assume that $a=0$ or $1$.  Then

\quad $s_1\to s_1,\ s_2\to s_3,\ s_3\to s_2$\quad when $a=0$,

\quad $s_1\to s_1,\ s_2\to s_2,\ s_3\to s_3s_2^{2\e}$\quad when $a=1$.

\noindent
are the desired isomorphisms.  The check is left to the reader.
\end{proof}

There are ten diffeomorphism classes of closed flat $3$-dimensional Riemannian manifolds; six orientable ones $\frakG_1,\frakG_2,\frakG_3,\frakG_4,\frakG_5,\frakG_6$ and four non-orientable ones $\frakB_1, \frakB_2,\frakB_3,\frakB_4$, see \cite[Theorems~3.5.5 and 3.5.9]{W}.  It is known that $\frakG_1,\frakG_2,\frakB_1,\frakB_3$ appear as real Bott manifolds (\cite{KM}, \cite{KN}).

\begin{Prop} \label{isoclass}
The isomorphism classes of $\pi_3=\Pi(a,\e,\e_1,\e_2)$ are classified into the following three types:
\begin{enumerate}
\item[$(1)$] Six Bieberbach groups:
\begin{center}
\begin{tabular}{||c|c|c||}
\hline%
$\Pi(a,\e,\e_1,\e_2)$& $a$ & $(\e,\e_1,\e_2)$  \\%
\hline\hline%
$\frakG_1$ & $0$ & $(1,1,1)$ \\%
\cline{1-3}%
$\frakG_2$ & $0$ & $(-1,-1,1)$ \\
\cline{1-3}%
$\frakB_1$ & {\rm even} & $(1,1,-1),(1,-1,1), (1,-1,-1),(-1,1,1)$ \\%
\cline{1-3}%
$\frakB_2$ & {\rm odd} & $(1,1,-1),(1,-1,1), (1,-1,-1),(-1,1,1)$\\%
\cline{1-3}%
$\frakB_3$ & {\rm even} & $(-1,1,-1), (-1,-1,-1)$ \\%
\cline{1-3}%
$\frakB_4$ & {\rm odd} & $(-1,1,-1), (-1,-1,-1)$ \\%
\hline %
\end{tabular}
\end{center}
\item[$(2)$] An infinite family of nilpotent groups
$$\Pi(a,1,1,1)\cong \Pi(-a,1,1,1) \text{ with $a\not=0$}.$$
\item[$(3)$] An infinite family of virtually nilpotent groups
$$\Pi(a,-1,-1,1)\cong \Pi(-a,-1,-1,1) \text{ with $a\not=0$}.$$
\end{enumerate}
\end{Prop}

\begin{proof}
(1)  First we note that groups $\Pi(a,\e,\e_1,\e_2)$'s for values of $(a,\e,\e_1,\e_2)$ in a same row in the table above are isomorphic to each other by Lemmas~\ref{lifting} and~\ref{4 iso}.
Three dimensional Bieberbach groups are classified and presented in \cite[Theorems~3.5.5 and 3.5.9]{W} with generators and relations, and we shall identify our groups $\Pi(a,\e,\e_1,\e_2)$ with them.

$(\frakG_1)$. Clearly $\Pi(0,1,1,1)$ is $\bbz^3$ and isomorphic to $\frakG_1$.

$(\frak G_2)$. Taking $\alpha=s_1, t_2=s_2$ and $t_3=s_3$, we see that $\Pi(0,-1,-1,1)$ is isomorphic to $\frakG_2$.

$(\frakB_1)$. We take
$
e=s_1,\ t_1=s_1^2,\ t_2=s_3,\ t_3=s_2^{-1}.
$
Then $\{e,t_1,t_2,t_3\}$ generates  $\Pi(0,-1,1,1)$ and satisfies
\[
e^2=t_1,\ e t_2e^{-1}=t_2,\ e t_3e^{-1}=t_3^{-1}.
\]
This shows that $\Pi(0,-1,1,1)$ is isomorphic to $\frakB_1$.

$(\frakB_2)$. We take
$
e=s_1,\ t_1=s_1^2,\ t_2=s_1^{-2}s_3,\ t_3=s_2.
$
Then $\{e,t_1,t_2,t_3\}$ generates $\Pi(1,-1,1,1)$ and satisfies
\[
e^2=t_1,\ e t_2e^{-1}=t_2,\ e t_3e^{-1}=t_1t_2t_3^{-1}.
\]
This shows that $\Pi(1,-1,1,1)$ is isomorphic to  $\frakB_2$.

$(\frakB_3)$.  We take
$
\alpha=s_1,\ e=s_2,\ t_1=s_1^2,\ t_2=s_2^2,\ t_3=s_3.
$
Then $\{\alpha,e,t_1,t_2,t_3\}$ generates $\Pi(0,-1,-1,-1)$ and satisfies
\begin{align*}
&\alpha^2=t_1,\ \alpha t_2\alpha^{-1}=t_2^{-1},\ \alpha t_3\alpha^{-1}=t_3^{-1},\\
&e^2=t_2,\ e t_1e^{-1}=t_1,\ e t_3e^{-1}=t_3^{-1},\ e\alpha e^{-1}=t_2\alpha.
\end{align*}
This shows that $\Pi(0,-1,-1,-1)$ is isomorphic to $\frakB_3$.

$(\frakB_4)$.  We take
$
\alpha=s_1,\ e=s_2,\ t_1=s_1^2,\ t_2=s_2^2,\ t_3=s_3^{-1}.
$
Then $\{\alpha,e,t_1,t_2,t_3\}$ generates $\Pi(1,-1,-1,-1)$ and satisfies
\begin{align*}
&\alpha^2=t_1,\ \alpha t_2\alpha^{-1}=t_2^{-1},\ \alpha t_3\alpha^{-1}=t_3^{-1},\\
&e^2=t_2,\ e t_1e^{-1}=t_1,\ e t_3 e^{-1}=t_3^{-1},\ e\alpha e^{-1}=t_2t_3\alpha.
\end{align*}
This shows that $\Pi(1,-1,-1,-1)$ is isomorphic to $\frakB_4$.

(2) The isomorphism in (2) of the proposition follows from Lemma~\ref{4 iso} (1).  Since the first homology group of $\Pi(a,1,1,1)$ is a cyclic group of order $|a|$, $\Pi(a,1,1,1)$ is isomorphic to $\Pi(b,1,1,1)$ if and only if $|a|=|b|$. For $\Pi(a,1,1,1)$ with $a\not=0$,
\[
[s_1,s_2]=s_3^a,\quad  [s_1,s_3]=[s_2,s_3]=1,
 \]
so the group has a central series
\[
\Pi(a,1,1,1)=\langle s_1,s_2,s_3\rangle \supset \langle s_3\rangle
\]
and hence is nilpotent.

(3) The isomorphism in (3) of the proposition also follows from Lemma~\ref{4 iso} (1).  For $\Pi(a,-1,-1,1)$ with $a\not=0$, we have
\[
[s_1^2,s_2]=s_3^{-2a},\quad  [s_1^2,s_3]=[s_2,s_3]=1.
\]
So the subgroup $H_a=\langle s_1^2,s_2,s_3\rangle$ of $\Pi(a,-1,-1,1)$ with $a\not=0$ is isomorphic to the nilpotent group $\Pi(2a,1,1,1)$ in (2).  Moreover, since the subgroup $H_a$ is of index 2, it is normal and the quotient group $\Pi(2a,1,1,1)/H_a$ is an order two cyclic group.  Therefore $H_a$ is the unique maximal nilpotent normal subgroup of $\Pi(2a,1,1,1)$ and $\Pi(2a,1,1,1)$ is virtually nilpotent.  Finally, if $\Pi(a,-1,-1,1)$ is isomorphic to $\Pi(b,-1,-1,1)$, then their maximal normal nilpotent subgroups $H_a$ and $H_b$ are isomorphic; so $|a|=|b|$ by (2) above.
\end{proof}

\begin{Rmk}
One can see that $\Pi(a,-1,-1,-1)$ with $a\not=0$ is isomorphic to an almost Bieberbach group (in short, an AB-group) of Seifert bundle type $3$ in \cite[Proposition~6.1]{DIKL}, or $\pi_3$ (the subscript $3$ also stands for Seifert bundle type $3$) in the list of \cite[p.~\!799]{CS}. Since the unique maximal normal nilpotent subgroup $H_a$ of $\Pi(a,-1,-1,-1)$ is isomorphic to $\Pi(2a,0,0,0)$, our class $\Pi(a,-1,-1,-1)$ consists of all of the infinitely many non-isomorphic AB-groups of type $3$.
\end{Rmk}

\subsection{Realization}
We shall observe that all the isomorphism classes of the groups $\Pi(a,\e,\e_1,\e_2)$ in Proposition~\ref{isoclass} can be realized by iterated $S^1$-bundles of height 3.

By Theorem~\ref{infra-nil}, the total space of an iterated $S^1$-bundle of height $3$ is a $3$-dimensional infra-nilmanifold.  The $3$-dimensional infra-nilmanifolds are well understood. In fact, these are ten flat Riemannian manifolds mentioned before or infra-nilmanifolds covered by the simply connected $3$-dimensional nilpotent Lie group $\Nil$, called the Heisenberg group,
\begin{equation} \label{Nil}
\Nil=\left\{\left[\begin{matrix}1&x&z\\0&1&y\\0&0&1\end{matrix}\right]\Big|\ x,y,z\in\bbr\right\}.
\end{equation}

As mentioned before, $\frakG_1,\frakG_2,\frakB_1,\frakB_3$ appear as real Bott manifolds (\cite{KM}, \cite{KN}), and in addition to them, $\frakB_2$ and $\frakB_4$ appear as iterated $S^1$-bundles as is shown in the following example.

\begin{Example}[Flat Riemannian manifolds of types $\frakB_2$ and $\frakB_4$]
\begin{equation*}
\begin{split}
s_1(x,y,z)&=(x+\frac{1}{2},\e y,-z+\frac{1}{4}), \\
s_2(x,y,z)&=(x,y+\frac{1}{2},-z), \\
s_3(x,y,z)&=(x,y,z+\frac{1}{2}).
\end{split}
\end{equation*}
The group generated by $s_1, s_2,s_3$ acts freely on $\bbr^3$ and has relations
\[
s_1s_2s_1^{-1}=s_3s_2^\e,\quad  s_{1}s_3s_{1}^{-1}=s_3^{-1},\quad  s_{2}s_3s_{2}^{-1}=s_3^{-1}.
\]
The subgroup generated by $s_1^2,s_2^2,s_3^2$ is the group $\bbz^3$ of translations.
The natural projections $\bbr^3\to\bbr^2\to\bbr$ induce an iterated $S^1$-bundle of height 3.
\end{Example}

The group $\Pi(a,1,1,1)$ in Proposition~\ref{isoclass} (2) can be realized by an iterated $S^1$-bundle as follows.

\begin{Example}[Nilmanifolds] \label{exam:nil}
It is well known that a lattice (i.e., a torsion free discrete cocompact subgroup) of $\Nil$ is isomorphic to
\begin{equation} \label{nilrel}
\Pi_a:=\Pi(a,1,1,1)=\left\langle s_1,s_2,s_3\mid [s_1,s_2]=s_3^a,\ [s_1,s_3]=[s_2,s_3]=1\right\rangle
\end{equation}
for some $a\ne0$. This group is realized as a lattice of $\Nil$ if one takes
$$
s_1=\left[\begin{matrix}1&1&0\\0&1&0\\0&0&1\end{matrix}\right],\
s_2=\left[\begin{matrix}1&0&0\\0&1&1\\0&0&1\end{matrix}\right],\
s_3=\left[\begin{matrix}1&0&\frac{1}{a}\\0&1&0\\0&0&1\end{matrix}\right].
$$
Then the orbit space $\Nil/\Pi_a$ is a nilmanifold with $\Pi_a$ as the fundamental group.

The product of the matrix in \eqref{Nil} with $s_i$  $(i=1,2,3)$ from the left is respectively given by
\[
\left[\begin{matrix}1&x+1&z+y\\0&1&y\\0&0&1\end{matrix}\right],\quad
\left[\begin{matrix}1&x&z\\0&1&y+1\\0&0&1\end{matrix}\right],\quad
\left[\begin{matrix}1&x&z+\frac{1}{a}\\0&1&y\\0&0&1\end{matrix}\right].
\]
Therefore, if we identify the matrix in \eqref{Nil} with the point $(x,y,z)$ in $\bbr^3$, then the left multiplication by $s_i$ on $\Nil$ for $i=1,2,3$ can  respectively be identified with the following diffeomorphism of $\bbr^3$:
\begin{equation*} \label{nils}
\begin{split}
s_1(x,y,z)&=(x+1,y,z+y), \\
s_2(x,y,z)&=(x,y+1,z), \\
s_3(x,y,z)&=(x,y,z+\frac{1}{a}).
\end{split}
\end{equation*}
So, the natural projections $\bbr^3\to\bbr^2\to\bbr$ induce an iterated $S^1$-bundle of height 3:
\[
\Nil/\Pi_a \to \bbr^2/\bbz^2=(S^1)^2\to\bbr/\bbz=S^1.
\]
Note that $\Nil/\Pi_a\to (S^1)^2$ above is the unit sphere bundle of an oriented plane bundle over $(S^1)^2$ whose Euler class is $a$ times a generator of $H^2((S^1)^2;\bbz)$.
\end{Example}

It is well known that all $3$-dimensional infra-nilmanifolds $M$ covered by $\Nil$ are Seifert manifolds (see \cite{O}); namely, $M$ is a circle bundle over a $2$-dimensional orbifold with singularities. It is known \cite[Proposition~6.1]{DIKL} that there are fifteen classes of distinct closed $3$-dimensional manifolds $M$ with a $\Nil$-geometry up to Seifert local invariant.

It is known (cf. \cite{CS,DIKL}) that the group $\aut(\Nil)$ of automorphisms of $\Nil$ is isomorphic to $\bbr^2\rtimes\GL(2,\bbr)$. In fact, an element
$$
\left(\left[\begin{matrix}u\\v\end{matrix}\right],
\left[\begin{matrix}a&b\\c&d\end{matrix}\right]\right)\in\bbr^2\rtimes\GL(2,\bbr)
$$
acts on $\Nil$ as follows:
$$
\left[\begin{matrix}1&x&z\\0&1&y\\0&0&1\end{matrix}\right]
\longmapsto
\left[\begin{matrix}1&ax+by&z'\\0&1&cx+dy\\0&0&1\end{matrix}\right],
$$
where
$$
z'=(ad-bc)z+\frac{1}{2}(acx^2+2bcxy+bdy^2)-(au+cv)x-(bu+dv)y.
$$
An infra-nilmanifold of dimension 3 is an orbit space of $\Nil$ by a cocompact discrete subgroup of the affine group $\aff(\Nil)=\Nil\rtimes\aut(\Nil)$ of $\Nil$ acting on $\Nil$ freely.

\begin{Example}[Infra-nilmanifolds] \label{exam:infranil}
Let $a\not=0$ as before.
We consider affine diffeomorphisms $s_1,s_2,s_3$ in $\aff(\Nil)$ given as follows:
\begin{alignat}{2}
s_1&=\left(\left[\begin{matrix}1&\tfrac{1}{2}&0\\0&1&0\\0&0&1\end{matrix}\right],
\left(\left[\begin{matrix}0\\0\end{matrix}\right], \left[\begin{matrix}1&\hspace{8pt}0\\0&-1\end{matrix}\right]\right)\right), \notag\\
s_2&=\left(\left[\begin{matrix}1&0&0\\0&1&\frac{1}{2}\\0&0&1\end{matrix}\right],
\left(\left[\begin{matrix}0\\0\end{matrix}\right], \left[\begin{matrix}1&0\\0&1\end{matrix}\right]\right)\right), \notag\\
s_3&=\left(\left[\begin{matrix}1&0&-\frac{1}{2a}\\0&1&\hspace{8pt}0\\0&0&\hspace{8pt}1\end{matrix}\right],
\left(\left[\begin{matrix}0\\0\end{matrix}\right], \left[\begin{matrix}1&0\\0&1\end{matrix}\right]\right)\right).\notag
\end{alignat}
In other words, if we identify $\Nil$ with $\bbr^3$ as before, then the diffeomorphisms $s_1,s_2,s_3$ are described as follows:
\begin{equation*} \label{infranils}
\begin{split}
s_1(x,y,z)&=(x+\frac{1}{2},-y,-z-\frac{y}{2}), \\
s_2(x,y,z)&=(x,y+1,z), \\
s_3(x,y,z)&=(x,y,z-\frac{1}{2a}).
\end{split}
\end{equation*}
One can check that the group $\Del_a$ generated by $s_1,s_2,s_3$ has relations
\begin{equation} \label{infranilrel}
s_{1}s_{2}s_{1}^{-1}=s_3^{a}s_{2}^{-1}, \quad s_{1}s_3s_{1}^{-1}=s_3^{-1},\quad  s_{2}s_3s_{2}^{-1}=s_3
\end{equation}
and the action of $\Del_a$ on $\Nil$ is free.  The subgroup of $\Del_a$ generated by $s_1^2,s_2,s_3$ agrees with $\Pi_{-2a}$ in Example~\ref{exam:nil} and $\Nil/\Pi_{-2a}\to \Nil/\Del_{a}$ is a double covering.  Note that the natural projections $\bbr^3\to\bbr^2\to\bbr$ induce an iterated $S^1$-bundle of height 3 with $\Nil/\Del_{a}$ as the total space.
\end{Example}

We summarize what we have observed as follows.

\begin{Thm} \label{3-dim class}
The total spaces of iterated $S^1$-bundles of height 3 are classified into the following three types up to homeomorphism:
\begin{enumerate}
\item[$(1)$] Closed flat Riemannian manifolds of types $\frakG_1, \frakG_2, \frakB_1, \frakB_2, \frakB_3, \frakB_4$.
\item[$(2)$] Nilmanifolds $\Nil/\Pi_a$ in Example~\ref{exam:nil} parametrized by positive integers $a$.
\item[$(3)$] Infra-nilmanifolds $\Nil/\Del_a$ in Example~\ref{exam:infranil} parametrized by positive integers $a$.
\end{enumerate}
\end{Thm}

\section{Flat Riemannian iterated $\bbr{P}^1$-bundles}  \label{sect6}

A real Bott manifold (see Example~\ref{real Bott}) is flat Riemannian although the total space of an iterated $\bbr{P}^1$-bundle is not necessarily flat Riemannian.  The purpose of this section is to show that real Bott manifolds are precisely flat Riemannian manifolds among the total spaces of iterated $\bbr{P}^1$-bundles.  In fact, we prove the following, which is essentially same as Theorem~\ref{main theo2} in the Introduction.

\begin{Thm} \label{theo}
Let $X_n$ be the total space of an iterated $\bbr{P}^1$-bundle of height $n$. If the fundamental group of $X_n$ is a Bieberbach group, then it is isomorphic to the fundamental group of a real Bott manifold.  $($This means that if $X_n$ is homeomorphic to a flat Riemannian manifold, then it is homeomorphic to a real Bott manifold.$)$
\end{Thm}

We consider the following setting:
\begin{equation} \label{1}
\begin{split}
s_is_js_i^{-1}&=s_n^{a_{ij}}s_j^{\e_{ij}} \quad \text{with $a_{ij}\in\bbz$, $\e_{ij}=\pm 1$ for $1\le i<j< n$},\\
s_is_ns_i^{-1}&=s_n^{\e_i} \quad\text{with $\e_i=\pm 1$ for $1\le i<n$}.
\end{split}
\end{equation}

\begin{Lemma} \label{lemm:1}
$s_i^2s_j^2=s_j^2s_i^2$ for $i<j$ if and only if $(\e_{i}+\e_{ij})(\e_{j}+1)a_{ij}=0$.
\end{Lemma}

\begin{proof}
The proof is essentially same as that in Lemma~\ref{3-flat}, so we omit it.
\end{proof}

\begin{Lemma} \label{lemm:2}
Fix $1\le k<n$ and suppose
\begin{equation} \label{3}
\text{$a_{ij}=0$ for all $i>k$}.
\end{equation}
Then, for $k<i<j<n$, we have
\begin{equation*}
\begin{split}
(\e_j-1)a_{ki}&=(\e_i-1)a_{kj} \quad\text{if $\e_{ij}=1$},\\
(\e_j-1)a_{ki}&=(\e_i+\e_j)a_{kj} \quad\text{if $\e_{ij}=-1$}.
\end{split}
\end{equation*}
\end{Lemma}

\begin{proof}
We conjugate the both sides of the former identity in \eqref{1} by $s_k$.  Then the left hand side turns into
\begin{equation} \label{4}
\begin{split}
s_k(s_is_js_i^{-1})s_k^{-1}&=(s_ks_is_k^{-1})(s_ks_js_k^{-1})(s_ks_i^{-1}s_k^{-1})\\
&=(s_n^{a_{ki}}s_i^{\e_{ki}})(s_n^{a_{kj}}s_j^{\e_{kj}})(s_n^{a_{ki}}s_i^{\e_{ki}})^{-1}\\
&=s_n^{a_{ki}+\e_i a_{kj}-\e_j a_{ki}}s_i^{\e_{ki}}s_j^{\e_{kj}}s_i^{-\e_{ki}}
\end{split}
\end{equation}
while since $a_{ij}=0$ for $i>k$ by assumption, the right hand side of \eqref{1} conjugated by $s_k$ turns into
\begin{equation} \label{5}
s_ks_j^{\e_{ij}}s_k^{-1}=\begin{cases} s_n^{a_{kj}}s_j^{\e_{kj}} \quad &\text{when $\e_{ij}=1$},\\
s_n^{-\e_ja_{kj}}s_j^{-\e_{kj}} \quad &\text{when $\e_{ij}=-1$}.
\end{cases}
\end{equation}
When $\e_{ij}=1$, $s_is_js_i^{-1}=s_j$ and hence $s_i^{\e_{ki}}s_j^{\e_{kj}}s_i^{-\e_{ki}}=s_j^{\e_{kj}}$.  Therefore, comparing exponents of $s_n$
at \eqref{4} and \eqref{5}, we obtain the former identity in the lemma. When $\e_{ij}=-1$, a similar argument yields the latter identity in the lemma.
\end{proof}

\begin{Lemma} \label{lemm:3}
Let \eqref{3} be satisfied and $k<i<j<n$ as in Lemma~\ref{lemm:2}.  Then the following hold.
\begin{enumerate}
\item[$(1)$] If $\e_i=\e_j=-1$, then $a_{ki}=a_{kj}$.
\item[$(2)$] If $\e_i=1$ and $a_{ki}\not=0$, then $\e_j=1$ and $\e_{\ell i}=\e_\ell$ for $k<\ell<i$.
\end{enumerate}
\end{Lemma}

\begin{proof}
(1) This is obvious from Lemma~\ref{lemm:2}.

(2) The first assertion $\e_j=1$ is obvious from Lemma~\ref{lemm:2}.  To prove the latter assertion, we apply Lemma~\ref{lemm:2} with $i=\ell$ and $j=i$.  When $\e_{\ell i}=1$, we have $a_{k \ell}(\e_i-1)=a_{ki}(\e_\ell-1)$.  This implies $\e_\ell=1$ because $\e_i=1$ and $a_{ki}\not=0$ by assumption. When $\e_{\ell i}=-1$, we have  $a_{k\ell}(\e_i-1)=a_{ki}(\e_\ell+\e_i)$.  This implies $\e_\ell=-1$ because $\e_i=1$ and $a_{ki}\not=0$.  In any case, $\e_{\ell i}=\e_\ell$.
\end{proof}

\begin{proof}[{Proof of Theorem~\ref{theo}}]
It suffices to prove that $\pi_1(X_n)$ is generated by $s_1,\dots,s_n$ with relations of the form \eqref{1} with $a_{ij}=0$ because the fundamental group of a real Bott manifold has such a presentation, see Example~\ref{real Bott}. We prove this assertion by induction on $n$.  The assertion is clearly true when $n=2$. When $n=3$, $\pi_1(X_3)$ is of the form \eqref{-1}, that is
\[
s_{1}s_{2}s_{1}^{-1}=s_3^{a}s_{2}^{\e}, \quad s_{1}s_3s_{1}^{-1}=s_3^{\e_{1}},\quad s_{2}s_3s_{2}^{-1}=s_3^{\e_{2}}.
\]
Here $a=0$ when $(\e,\e_1,\e_2)=(1,1,1)$ or $(-1,-1,1)$ by Lemma~\ref{3-flat} and $a$ is even otherwise by Proposition~\ref{char of RP} (2).  Therefore one can assume $a=0$ by Lemma~\ref{lifting}, so the assertion is true when $n=3$.

Now we assume that the assertion is true for $\pi_1(X_{n-1})$ with some $n\ge 4$.  Then $\pi_1(X_n)$ is generated by $s_1,\dots,s_n$ with relations of the form \eqref{1}.  We shall show that we can achieve $a_{ij}=0$ by replacing $s_i$ $(1\le i<n)$ by $s_n^{b_i}s_i$ with suitable $b_i\in\bbz$.

First we look at the following (last) three relations among {$s_{n-2},s_{n-1},s_{n}$}:
\[
s_{n-2}s_{n-1}s_{n-2}^{-1}=s_n^{a}s_{n-1}^{\e}, \quad s_{n-2}s_ns_{n-2}^{-1}=s_n^{\e_{n-2}},\quad s_{n-1}s_ns_{n-1}^{-1}=s_n^{\e_{n-1}}
\]
where $a=a_{n-2\ n-1}$ and $\e=\e_{n-2\ n-1}$.  Since $X_n$ is an iterated $\bbr P^1$-bundle, one can assume $a=0$ by the same reason as the case $n=3$.

Now suppose that for some $k<n-2$, we have achieved $a_{ij}=0$ for all $i>k$; so we are under the situation of Lemmas~\ref{lemm:2} and~\ref{lemm:3}.  What we shall prove is that we can achieve $a_{ij}=0$ for all $i\ge k$.  Let $p>k$.  If $a_{kp}=0$, we have nothing to do; so we assume $a_{kp}\not=0$.  We distinguish two cases according to the value of $\e_p$.

The case where $\e_p=-1$.  We replace $s_k$ by $s_n^bs_k$.  This replacement does not affect relations for $s_i$ and $s_j$ with $k<i<j$, so it keeps $a_{ij}=0$ for $k<i<j$.  But the relation $s_ks_ps_k^{-1}=s_n^{a_{kp}}s_p^{\e_{kp}}$ turns into
\[
s_ks_ps_k^{-1}=s_n^{a_{kp}+2b}s_p^{\e_{kp}}
\]
because $\e_p=-1$.  Since $a_{kp}$ is even by Proposition~\ref{char of RP} (2), one can take $b=-a_{kp}/2$ so that the exponent of $s_n$ above becomes zero.  For other $q>k$ with $\e_q=-1$, $a_{kq}=a_{kp}$ by Lemma~\ref{lemm:3} (1).  Therefore, $b$ is independent of $p$ with $\e_p=-1$.

The case where $\e_p=1$.  We note $\e_j=1$ for any $j>k$ by Lemma~\ref{lemm:3} (2).  We replace $s_p$ by $s_n^{c}s_p$.  In this case, it is not obvious that this replacement keeps $a_{ij}=0$ for $k<i<j$ but it does.  In fact, its suffices to check that the relations $s_ps_js_p^{-1}=s_j^{\e_{pj}}$ for $p<j$ and $s_\ell s_ps_\ell^{-1}=s_p^{\e_{\ell p}}$ for $\ell <p$ remain unchanged and one can easily check that the former identity remains unchanged because $\e_j=1$ and the latter one remains unchanged because $\e_p=1$ by assumption and $\e_{\ell p}=\e_\ell$ by Lemma~\ref{lemm:3} (2).
However, the relation $s_ks_ps_k^{-1}=s_n^{a_{kp}}s_p^{\e_{kp}}$ turns into
\begin{equation} \label{6}
s_ks_ps_k^{-1}=s_n^{a_{kp}+(\e_k-\e_{kp})c}s_p^{\e_{kp}}.
\end{equation}
On the other hand, since $s_k^2s_p^2=s_p^2s_k^2$, we have $(\e_k+\e_{kp})(\e_p+1)a_{kp}=0$ by Lemma~\ref{lemm:1}.  Since $a_{kp}\not=0$ and $\e_p=1$, this implies $\e_k=-\e_{kp}$ and hence $\e_k-\e_{kp}=\pm 2$.  Since $a_{kp}$ is even by Proposition~\ref{char of RP} (2), one can take $c=-a_{kp}/(\e_k-\e_{kp})$ so that the exponent of $s_n$ becomes zero in \eqref{6}.

In any case, we can achieve $a_{kp}=0$ for any $p>k$ keeping $a_{ij}=0$ for $k<i<j$.  This completes the induction step and proves the theorem.
\end{proof}

\bigskip

\noindent
{\bf Acknowledgment.}  The authors would like to thank Y. Kamishima for his helpful comments on a preliminary note of the paper.  In particular, the proof of Lemma~\ref{Bieberbach} has been greatly simplified and we have learned that a similar study to our paper is done in \cite{N}.
Finally, the authors would like to thank the anonymous referee for helpful comments.

\end{document}